\documentclass[11pt,a4paper]{amsart}
\usepackage{lscape}
\usepackage{hyperref} 
\usepackage{multirow}
\usepackage{amsmath,amsthm, amscd, amssymb, amsfonts}
\usepackage{epsfig}
\usepackage{amsmath,amsthm,amssymb, amscd,enumerate}

\usepackage{pdfsync}
\numberwithin{equation}{section}
\theoremstyle{plain}
\newtheorem{theorem} {Theorem} [section]

\newtheorem{lemma} [theorem] {Lemma}
\newtheorem{corollary} [theorem] {Corollary}
\newtheorem{proposition} [theorem] {Proposition}

\theoremstyle{definition}

\newtheorem{definition}[theorem]{Definition}

\newtheorem{problem}[theorem]{Problem}

\renewcommand \parallel {/\kern-3pt/}
\newcommand \N {\mathbb N}
\newcommand \Z {\mathbb Z}

\newcommand \z {\mathfrak{z}}
\newcommand \n {\mathfrak{n}}
\newcommand \gl {\mathfrak{gl}}

\newcommand \s {\mathfrak{s}}

\renewcommand \k {\mathrm{k}}

  {\centering\normalfont\Large\bfseries}

\begin{document}


\title[A lower bound for faithful repns of nilpotent Lie algebras]{A lower bound for faithful representations of nilpotent Lie algebras}

\author{Leandro Cagliero}
\address{CONICET - FaMAF, Universidad Nacional de C\'ordoba, Argentina}
\email{cagliero@famaf.unc.edu.ar}

\author{Nadina Rojas}
\address{CIEM, FCEFyN Universidad Nacional de C\'ordoba,  Argentina}
\email{ nrojas@efn.uncor.edu}

\thanks{Partially supported by CONICET, FONCyT and SECyT-UNC Grants (Argentina)}



\begin{abstract}
In this paper we present a lower bound for the minimal dimension $\mu(\n)$
of a faithful representation of a finite dimensional $p$-step nilpotent Lie algebra $\n$ over a field of characteristic zero.
Our bound is given as the minimum of a quadratically constrained linear optimization problem,
it works for arbitrary $p$ and takes into account a given filtration of $\n$.
We present some estimates of this minimum which leads to a very explicit
lower bound for $\mu(\n)$ that involves the dimensions of $\n$ and its center.
This bound allows us to obtain $\mu(\n)$ for some families of nilpotent Lie algebras.
\end{abstract}

\maketitle


\section{Introduction and main results}\label{intro}


In this paper all Lie algebras and representations are finite dimensional over field
$\k$ of characteristic zero.
Given a representation $(\pi, V)$ of a nilpotent Lie algebra
$\mathfrak{n}$, we say that
$(\pi, V)$ is a \emph{nilrepresentation} if
$\pi(X)$ is a nilpotent for all $X \in \mathfrak{n}$.

Ado's Theorem states that any  Lie algebra
has a faithful representation (see \cite[p. 202]{J}).
Nevertheless, given a Lie algebra $\n$, the invariants
\begin{align*}
\mu(\mathfrak{n}) &= \min \{\dim V : (\pi, V) \text{ is a faithful representation of } \mathfrak{n}\}, \\
\mu_{nil}(\mathfrak{n})&= \min \{\dim V : (\pi, V) \text{ is a faithful nilrepresentation of } \mathfrak{n}\}.
\end{align*}
are, in general, very difficult to compute or even to estimate.
Apart from its intrinsic interest,
the map $\mu$ is not only important in computational mathematics,
but it is also connected to the theory of compact affine manifolds and crystallographic groups
(see for instance \cite{Be,Bu,K,Mi,Se})
and to the theory of polycyclic groups
(see for instance \cite{Se1}[Ch. 5,6], \cite{GSe}[\S3.2]).

 The value of $\mu(\n)$ has been obtained only for very few
families of Lie algebras $\n$ (see, for instance \cite{Be, Bu, BM, CR, R, S}).

Obtaining general results about $\mu$, in particular new bounds, is very hard.
On the one hand, there is a number of papers investigating new methods for
constructing faithful representations of small dimension for a given class of (nilpotent) Lie algebras
(see for instance \cite{BM2,BEdG,dG,dGN,Ne})
and thus obtaining upper bounds for $\mu$.
In this direction, an ambitious goal is to find out whether there is a fixed polynomial $p$
such that $\mu(\n)\le p(\dim\n)$ for all Lie algebras $\n$ (at least inside a wide class).

On the other hand, general lower bounds are crucial for proving
that a given faithful representation of a Lie algebra is actually of minimal dimension.
They are also important for their applications to other problems.
For instance, the counterexample obtained by
Benoist \cite{Be} to Milnor's conjecture \cite{Mi} is based on a family of
Lie algebras satisfying $\mu(\n)>\dim\n+1$.
On the group theory side, lower bounds for faithful representations of finite groups have been used
to obtain a lower bound for the smallest non-trivial eigenvalue of the Laplace-Beltrami operator
on certain manifolds
\cite{SX},
or to answer questions of Lubotzky about the
uniform expansion bounds for the Cayley graphs of $\text{SL}_2(\mathbb{F}_p)$
\cite{BoG}.

\smallskip

In this paper we obtain the following lower bound of $\mu_{nil}$ for nilpotent Lie algebras.

\begin{theorem}\label{thm.main}
Let
$\mathfrak{n}$ be a Lie algebra and let
$\mathfrak{n}_p \subset \dots \subset \mathfrak{n}_1=\mathfrak{n}$ be a filtration of
$\mathfrak{n}$ such that $\mathfrak{n}_{p_0}$ is contained in the center of $\n$.
Then
\[
 \mu_{nil}(\mathfrak{n})\ge r_0^{\text{min}}
\]
where $r_0^{\text{min}}$ is the minimum value of
\[
 r_0 = a_0+a_1+ \dots +a_p,\quad a_0,a_1,\dots ,a_p\in \Z,
\]
subject to the following restrictions:

\medskip
\noindent
\begin{tabular}{lll}
(a) & $a_0,a_p\ge 1$ and $a_k\ge 0$, & for $k=1,\dots,p-1$;  \\[2mm]

(b) & $\displaystyle\sum_{i=0}^{p_0-k} a_i (a_{k + i}+\dots+a_p)\ge \dim{\n}_k $, & for  $k= 1, \dots, p_0$; \\[6mm]

(c) & $a_0  (a_{k}+\dots+a_p)  \geq  \dim{\n}_k$, & for    $k= p_0, \dots, p$.
\end{tabular}
\end{theorem}

The quadratically constrained linear optimization problem involved in the above theorem seems to be difficult.
In this paper we present some quick, but not trivial,  estimations of $r_0^{\text{min}}$
and the lower bounds obtained are already interesting.
We are confident that future research on $r_0^{\text{min}}$
will provide very good lower bounds for $\mu_{nil}(\mathfrak{n})$.
As a consequence of our estimates, we obtain the following theorem.

\begin{theorem}\label{Thm.mainbound}
 Let $\mathfrak{n}$ be a
 $p$-step nilpotent Lie algebra, $p>1$, and let
 $\z$ be the center of
 $\mathfrak{n}$.

 \smallskip
\begin{enumerate}[(1)]
 \item If $\dim\n \ge \big((p-1)^2+p^2\big)\dim\z$ then
\[
   \mu_{nil}(\mathfrak{n})\ge  \sqrt{\frac{2p}{p-1}(\dim\n-\dim\z)}.
\]
 \item If $\dim\n \le \big((p-1)^2+p^2\big)\dim\z$ then
\[
   \mu_{nil}(\mathfrak{n})\ge
   \sqrt{\frac{2(p-1)}{p-2}\dim\n +\frac{2p(p-1)}{(p-2)^2}\dim\z}\,-\,\frac{2}{p-2}\sqrt{\dim\z},
\]
if $p\ne 2$, and   $ \mu_{nil}(\mathfrak{n})\ge \frac{\dim\n+3\dim\z}{2\sqrt{\dim\z}}$ if
$p=2$.
\end{enumerate}
In both cases, the given bound is bigger than $\sqrt{ \frac{2(p+1)}{p} \dim \mathfrak{n}}$.
\end{theorem}
From this theorem, $\mu_{nil}$ is obtained for the following families.
\begin{enumerate}[(i)]
 \item Given
$p, a \in \mathbb{N}$, let
\begin{equation*}\label{eq:2}
{\mathfrak{n}}_{a,p} = \left\{
          \left(
          \begin{smallmatrix}
          0 & A_{12} & A_{13} & \dots & A_{1p+1} \\
            & 0 & A_{23} & \dots & A_{2p+1} \\
            &   &   \ddots     &       & \vdots\\
            &  0  &        &       & A_{pp+1} \\
            &    &        &       & 0
          \end{smallmatrix}
          \right) : A_{ij} \in M_a(\k) \text{ para } 1 \leq i < j \leq p+1
          \right\},
\end{equation*}
then 
$\mu(\mathfrak{n}_{a,p})= (p + 1)a$.

\medskip

\item Given
$a, b, c \in \mathbb{N}$  let
$$
\mathfrak{n}_{a,b,c}= \left\{
            \left(\begin{smallmatrix}
          0 & A_{ab} & A_{ac}\\
            & 0 & A_{bc}\\
          &   & 0
          \end{smallmatrix}
          \right) : A_{ab} \in M_{a,b}(\k),A_{ac} \in M_{a,c}(\k), A_{bc}(\k) \in M_{b,c}\right\}.
$$
Then, if either
$b=a+c$, or
$a=c$ and
$b\le 2a$, 
$$
\mu(\mathfrak{n}_{a,b,c}) = a+b+c.
$$

\end{enumerate}
The above two families are nilradicals of parabolic subalgebras of simple Lie algebras of type A.
The above result shows that their defining representation is faithful of minimal dimension.
However this is not true for all nilradicals of type A.
For instance
if 
$a=b=1$, then the Lie algebra 
$\n_{1,1,c}$ given in (ii) satisfies
$\mu(\mathfrak{n}_{1,1,c})= \left \lceil 2\sqrt{2c} \right\rceil < 2 + c$ for all $c\in\N$,
as shown in \cite{AR}.

\medskip

The paper is organized as follows. In \S\ref{preli} we prove Theorem \ref{teo:descomposicion} which is a key result.
It allows us to obtain certain special bases for faithful
representations of nilpotent Lie algebras
that eventually lead, in \S\ref{lower}, to the optimization problem of Theorem \ref{thm.main}.
In this section, an open question is posed.
In \S\ref{sec.aplications}
we compute $\mu_{nil}$ for the families (i) and (ii). 
In \S\ref{sec.estimates} we obtain estimates for the minimum of
our optimization problem and prove Theorem \ref{Thm.mainbound}.


\section{Linearly independent subsets associated to chains of endomorphisms}\label{preli}


In this section we describe an algorithm that, given a faithful $\n$-module $V$,
will provide a basis of $V$ with certain special properties that
will allow us to estimate $\dim V$.

First, we recall the following standard lemma.

\begin{lemma}\label{lemma:maximo}
Let
$V$ be a vector space and let
$\mathcal{T}_1, \dots, \mathcal{T}_p$ be vector subspaces of
$End(V)$.
If
$r_i= \max \{\dim \mathcal{T}_i v : v \in V\}$ and
$W_i= \{w \in V : \dim \mathcal{T}_i w= r_i\}$,
then $\cap_{i=1}^p W_i$
is a non-empty open dense subset of $V$.
In particular, there exists
$v \in V$ such that
$
\dim \mathcal{T}_i v = r_i$
for all
$i=1, \dots, p$.
\end{lemma}

\begin{proof}
Since the intersection of open dense subsets is
a non-empty open dense subset, it suffices to prove that $W_i$
is open and dense for all $i$.
Let us fix $i=1,\dots,p$, and let
$w \in W_i$ and let
$\{T_1, \dots, T_{r_i} \} \subseteq \mathcal{T}_i$ be such that
$$
\{T_1(w), \dots, T_{r_i}(w)\}
$$
is a basis of
$\mathcal{T}_iw$. For any
$v\in V$,
$A(v)$ denote the matrix whose columns are the coordinates of
$T_1(v), \dots, T_{r_i}(v)$ in a given basis
$B$ of $V$. Since
$\{T_1(w), \dots, T_r(w)\}$ is a linearly independent set, the matrix
$A(v)$ has an
$(r \times r)$-minor $a(v)$ such that $\det a(w)\ne0$.
Therefore, the open set
$U=\{v\in V:\det a(v)\ne0\}$ contains $w$,
is contained in $W_i$ and, since $\k$ is an infinite field, it is dense.
\end{proof}

\begin{definition}
Given a vector space $V$ and a sequence of vector subspaces
$\mathcal{T}_1, \dots, \mathcal{T}_p$ of
$\text{End}(V)$,
we say that $v\in V$ is \emph{rank-vector} for the sequence $\mathcal{T}_1, \dots, \mathcal{T}_p$,
if
\[
\dim \mathcal{T}_i v = \max \{\dim \mathcal{T}_i w : w \in V\}
\]
for all
$i=1, \dots, p$.
\end{definition}

Let $\mathcal{T}_p \subset \dots \subset  \mathcal{T}_1= \mathcal{T}$ be a chain of
vector subspaces of $\text{End}(V)$ and let $\{v_1,v_2,v_3, \dots\}$ the sequence (which eventually will be finite) 
obtained by applying the following procedure:
\begin{enumerate}[(1)]
 \item choose a rank-vector
       $v_1$ for the chain $\mathcal{T}$,
 \item choose a (special) linear complement  $\mathcal{T}'$ of the annihilator of $v_1$ in $\mathcal{T}$,
 \item choose a rank-vector $v_2$ for the chain $\mathcal{T}'$,
\end{enumerate}
  and so on.
More precisely, the procedure is given by the following algorithm.

\medskip

\begin{enumerate}[(i)]
 \item \texttt{For all $k=1,\dots,p$, let $s_k:=0$ and $\mathcal{R}_{k} := \mathcal{T}_k$.
  \\  Let $i:=0$, $q:=p$.}
\medskip
 \item  \texttt{Increase $i$ by 1.}
\medskip
 \item  \texttt{Let $v_i$ be a rank-vector associated to $\mathcal{R}_q \subset \dots \subset  \mathcal{R}_1$.}
\medskip
 \item  \texttt{For all $k=1,\dots,q$, let
\[
 \tilde{\mathcal{R}}_{k} = \text{Ann}_{\mathcal{R}_k}(v_i)=\{T\in\mathcal{R}_k: T(v_i)=0\},
\]
If $\tilde{\mathcal{R}}_{k}\ne\mathcal{R}_k$, increase $s_k$ by 1 and
 let $\mathcal{T}_{k,i}$ be such that
\[
 \mathcal{R}_{k}=\mathcal{T}_{k,i}\oplus \tilde{\mathcal{R}}_{k} \text{ and }
 \mathcal{T}_{k,i} \supseteq \mathcal{T}_{k+1,i}\quad \text{(assume $\mathcal{T}_{q+1,i}=0$)}.
\]}
 \item  \texttt{If $\tilde{\mathcal{R}}_{1}\ne 0$,  let $q$ be the largest $j$ such that $\tilde{\mathcal{R}}_{j}\ne0$ and let
\[
\mathcal{R}_{k} := \tilde{\mathcal{R}}_{ k},\quad\text{ $k=1,\dots,q$,}
\]
(we have  $\mathcal{R}_{q} \subset \dots \subset  \mathcal{R}_1$). Go to (ii).}
\medskip
\item   \texttt{End.}
\end{enumerate}
\bigskip
As a result we obtain:
\begin{enumerate}[(a)]
 \item A partition $s_1 \geq s_2\geq \dots \geq s_p>0$, ($s_1$ is the final value of $i$).

\medskip
\item A set $\{v_1, v_2,\dots, v_{s_1}\}$.

\medskip
 \item A family of subspaces $\mathcal{T}_{k,j}\subset\text{End}(V)$, $1\le j\le s_k$ and $1\le k\le p$.
\end{enumerate}

\medskip

The following theorem summarizes some of the main properties
of the set $\{v_1, v_2,\dots, v_{s_1}\}$
and the family of subspaces $\mathcal{T}_{k,j}\subset\text{End}(V)$.

\begin{theorem}\label{teo:descomposicion}
Let $V$ be a vector space and let
$\mathcal{T}_p \subset \dots \subset  \mathcal{T}_1$ be
a chain of subspaces in $\text{End}(V)$.
Then
there exist
a partition $s_1 \geq s_2\geq \dots \geq s_p>0$,
a linearly independent set
$\{v_1, \dots, v_{s_1}\} \subset V$
and a family of subspaces $\mathcal{T}_{k,j}\subset\text{End}(V)$, $1\le j\le s_k$ and $1\le k\le p$,
such that:

\medskip

\noindent
$$
\begin{array}{cccccccccccc}
\!\!\!\text{(1)} \!&\!
{\mathcal{T}}_1\!&\!=\!&\!{\mathcal{T}}_{1,1} \!&\! \oplus  \dots  \oplus \!&\! {\mathcal{T}}_{1, s_{ p}}\!&\! \oplus  \dots \oplus \!&\! {\mathcal{T}}_{1, s_{ p-1}} \!&\! \oplus  \dots  \oplus \!&\! {\mathcal{T}}_{1, s_2} \!&\! \oplus  \dots  \oplus \!&\! {\mathcal{T}}_{1, s_1}\\
 \!&\!   \cup       \!&\! \!&\! \cup                               \!&\!          \!&\!  \cup          \!&\!                              \!&\!    \cup                 \!&\!                             \!&\!      \cup       \\
\!&\!{\mathcal{T}}_2\!&\!=\!&\!{\mathcal{T}}_{2,1} \!&\! \oplus  \dots  \oplus \!&\! {\mathcal{T}}_{2, s_{ p}}\!&\! \oplus   \dots  \oplus \!&\! {\mathcal{T}}_{2, s_{ p-1}} \!&\! \oplus  \dots  \oplus \!&\! {\mathcal{T}}_{2, s_2}\\
 \!&\!   \cup       \!&\! \!&\! \cup                               \!&\!          \!&\!  \cup          \!&\!                              \!&\!    \cup                 \!&\!\\
\!&\!\vdots         \!&\!   \!&\!  \vdots           \!&\!                         \!&\!   \vdots      \!&\!                                \!&\!       \vdots    \\
 \!&\!   \cup       \!&\! \!&\! \cup                               \!&\!          \!&\!  \cup          \!&\!                              \!&\!    \cup                 \!&\!\\
\!&\!{\mathcal{T}}_{p-1}\!&\!=\!&\!{\mathcal{T}}_{{p-1},1} \!&\! \oplus  \dots  \oplus \!&\! {\mathcal{T}}_{{p-1}, s_{ p}}\!&\! \oplus   \dots  \oplus \!&\! {\mathcal{T}}_{{p-1}, s_{ p-1}}\\
\!&\!    \cup       \!&\! \!&\! \cup             \!&\!                          \!&\!  \cup          \!&\!                             \!&\!     \\
\!&\!{\mathcal{T}}_{ p}\!&\!=\!&\!{\mathcal{T}}_{ p,1} \!&\! \oplus  \dots  \oplus \!&\! {\mathcal{T}}_{ p, s_{ p}}
\end{array}
$$

\noindent
We notice that this display resembles the Young diagram of the partition
$s_1 \geq s_2\geq \dots \geq s_p$.

\smallskip
\begin{enumerate}[(2)]
\item[(2)] $\dim \mathcal{T}_{k,j}= \dim \mathcal{T}_{k,j} v_j$ for
      $j=1, \dots, s_k$ and $k= 1, \dots, p$.
\smallskip
\item[(3)] $\mathcal{T}_{k,j}v_i= 0$ for
      $1 \leq i < j \leq s_k$ and $k=1, \dots, p$.
\smallskip
\item[(4)] $\mathcal{T}_{k,j}V \subseteq \mathcal{T}_{k,i}v_i$ for $1 \leq i < j \leq s_k$ and $k=1, \dots, p$.
\end{enumerate}
Moreover, if
$\mathcal{T}_1$ consists of nilpotent operators and
$[\mathcal{T}_1, \mathcal{T}_{p_0}]= 0$,
$1\leq p_0 \leq p$, then
$\mathcal{T}_{1,1} v_1 \cap \operatorname{span}_{\k}\{v_1, \dots, v_{s_{p_0}}\}= 0$.
\end{theorem}

\begin{proof}
By construction, it is clear that properties (1), (2) and (3) hold.

 We first prove that $\{v_1, v_2,\dots, v_{s_1}\}$
 is linearly independent.
By construction, we may assume, as an induction hypothesis, that $\{v_2,\dots, v_{s_1}\}$
is linearly independent. Thus we must show that $v_1 \notin  \operatorname{span}_{\k}\{v_2,\dots, v_{s_1}\}$.

If
\begin{equation}\label{eq.v1}
v_1= \sum_{j=2}^{ s_1} a_{j}  v_j
\end{equation}
let $j_0= \max \{j : a_j \neq 0\}\ge2$ and let
$T \in {\mathcal{T}}_{1,j_0} \subset {\mathcal{T}}_1$, $T\ne0$.
We now apply $T$ to both sides of \eqref{eq.v1}.
Property (3) implies that the left-hand side is zero and the right hand side is
$a_{j_0} T( v_{j_0})$.
On the other hand, property (2) says that
$T( v_{j_0})\ne 0$, which is a contradiction.

We now prove (4).
If
$s_1=1$ then
$\mathcal{T}_{k,1}= \mathcal{T}_k$ for all
$k= 1, \dots,p$ and  condition (4) is empty.
As we did earlier, we may assume by induction
that $\mathcal{T}_{k,j}V \subseteq \mathcal{T}_{k,i}v_i$ for $2 \leq i < j \leq s_k$ and $k=1, \dots, p$.
Thus, we only need to prove (4) when $i= 1$.
This is equivalent to prove that
$T(v) \in \mathcal{T}_{k,1}v_1$ for all $v\in V$ and  all
$T \in \mathcal{T}_{k,2} \oplus \dots \oplus \mathcal{T}_{k,s_k}$, $k= 1, \dots, p$.
If $r_k=\dim \mathcal{T}_{k,1}v_1$ and
$\{T_1, \dots, T_{r_k}\}$ is a basis of
$\mathcal{T}_{k,1}$, we must show that
\[
\{T(v),T_1(v_1), \dots, T_{r_k}(v_1)\}
\]
is linearly dependent for all
$T \in \mathcal{T}_{k,2} \oplus \dots \oplus \mathcal{T}_{k,s_k}$,
$k= 1, \dots, p$,  and all
$v\in V$.

Let us fix such $T$, $v$ and $k$. By the definition of
$v_1$, the set
\[
\{T(v_1 + tv), T_1(v_1 + tv), \dots,  T_{r_k}(v_1 + tv)\}
\]
 is linearly dependent for all
$t \in \k$. Since
$T(v_1)= 0$, we obtain that
$\{T(v), T_1(v_1 + tv), \dots, T_{r_k}(v_1 + tv)\}$ is linearly dependent for all $t \neq 0$.
 Since  $\k$ infinite, we conclude that this last set is linearly dependent for
$t= 0$.
This completes the proof of (4).

We now prove the `moreover' part of the theorem.
We must show that
\[
 \mathcal{T}_{1,1} v_1 \cap \operatorname{span}_{\k}\{v_1, \dots, v_{s_{p_0}}\}= 0.
\]
Suppose, on the contrary, that
 there exist $T\in\mathcal{T}_{1,1}$, $T\ne0$, and
$a_1, \dots, a_{s_{p_0}} \in \k$ such that
\begin{equation}\label{eq:5}
T(v_1) = \sum_{j=0}^{s_{p_0}} a_j v_j.
\end{equation}
Since $T\in\mathcal{T}_{1,1}$ and $T\ne0$, it follows that $T(v_1)\ne0$
and thus $a_j \neq 0$ for some $j$.
Let
$j_0= \max\{j : a_j \neq 0\}$.
Since
$T$ is nilpotent, its only eigenvalue is zero and thus
$1<j_0\le s_{p_0}$.

Let
$T' \in \mathcal{T}_{p_0,j_0}$, $T'\ne0$, and let us apply
$T'$ to both sides of (\ref{eq:5}).
Since
$T' \in \mathcal{T}_{p_0,j_0}$ and $j_0>1$
we obtain on the left hand side
$T'T(v_1)= TT'(v_1)= 0$.
On the other hand, it follows from properties (2) and (3) that the right hand side is
$a_{j_0}T'(v_{j_0}) \neq 0$, which is a contradiction.
\end{proof}


\section{An optimization problem leading to a lower bound for \texorpdfstring{$\mu_{nil}$}{mu nil}}\label{lower}


Let $V$ be a vector space and let
$\mathfrak{n}$ be a Lie subalgebra of
$\gl(V)$ consisting of nilpotent endomorphisms.
Let
\[
\mathfrak{n}_p \subset \dots \subset \mathfrak{n}_2 \subset \mathfrak{n}_1= \mathfrak{n}
\]
be a filtration of
$\mathfrak{n}$ such that
$\mathfrak{n}_{p_0}$
is contained in the center of $\n$ for some $1 \leq p_0 \leq p$.

Applying Theorem \ref{teo:descomposicion} to the filtration
$\mathfrak{n}_p \subset \dots \subset \mathfrak{n}_2 \subset \mathfrak{n}_1= \mathfrak{n}$
we obtain a partition $s_1 \geq s_2\geq \dots \geq s_p>0$,
a linearly independent set
$\{v_1, \dots, v_{s_1}\} \subset V$,
and a decomposition
$$
\begin{array}{ccccccccccc}
{\mathfrak{n}}_1\!&\!=\!&\!{\mathfrak{n}}_{1,1} \!&\! \oplus  \dots  \oplus \!&\! {\mathfrak{n}}_{1, s_{ p}}\!&\! \oplus  \dots \oplus \!&\! {\mathfrak{n}}_{1, s_{ p-1}} \!&\! \oplus  \dots  \oplus \!&\! {\mathfrak{n}}_{1, s_2} \!&\! \oplus  \dots  \oplus \!&\! {\mathfrak{n}}_{1, s_1}\\
   \cup       \!&\! \!&\! \cup                               \!&\!          \!&\!  \cup          \!&\!                              \!&\!    \cup                 \!&\!                             \!&\!      \cup       \\
{\mathfrak{n}}_2\!&\!=\!&\!{\mathfrak{n}}_{2,1} \!&\! \oplus  \dots  \oplus \!&\! {\mathfrak{n}}_{2, s_{ p}}\!&\! \oplus   \dots  \oplus \!&\! {\mathfrak{n}}_{2, s_{ p-1}} \!&\! \oplus  \dots  \oplus \!&\! {\mathfrak{n}}_{2, s_2}\\
 \cup \!&\! \!&\!    \cup       \!&\! \!&\! \cup                               \!&\!          \!&\!  \cup          \!&\!                              \\
\vdots \!&\!\!&\!\vdots         \!&\!   \!&\!  \vdots           \!&\!                         \!&\!   \vdots      \!&\!    \\
  \cup \!&\! \!&\!   \cup       \!&\! \!&\! \cup                               \!&\!          \!&\!  \cup          \!&\!                             \\
{\mathfrak{n}}_{p-1}\!&\!=\!&\!{\mathfrak{n}}_{{p-1},1} \!&\! \oplus  \dots  \oplus \!&\! {\mathfrak{n}}_{{p-1}, s_{ p}}\!&\! \oplus   \dots  \oplus \!&\! {\mathfrak{n}}_{{p-1}, s_{ p-1}}\\
 \cup \!&\! \!&\!   \cup       \!&\! \!&\! \cup             \!&\!                          \!&\!  \cup          \!&\!                             \!&\!     \\
{\mathfrak{n}}_{ p}\!&\!=\!&\!{\mathfrak{n}}_{ p,1} \!&\! \oplus  \dots  \oplus \!&\! {\mathfrak{n}}_{ p, s_{ p}}
\end{array}
$$
such that
\begin{equation}\label{eq:12}
\mathfrak{n}_{1,1}v_1 \cap \operatorname{span}_{\k}\{v_1, \dots, v_{s_{p_0}}\}=0.
\end{equation}

Let $r_k=\dim\mathfrak{n}_{k,1}$. Since
$\mathfrak{n}_{k,1} \subseteq \mathfrak{n}_{k-1,1}$, there exists a basis
$\{X_1, \dots, X_{r_1}\}$ of
$\mathfrak{n}_{1,1}$ such that
$\{X_1, \dots, X_{r_k}\}$ is a basis of
$\mathfrak{n}_{k,1}$,
$k=1, \dots, p$.

It follows from Theorem \ref{teo:descomposicion} that
\begin{equation}\label{eq:1}
\{X_1(v_1), \dots, X_{r_k}(v_1)\}
\end{equation}
is a basis of $\mathfrak{n}_{k,1}v_1$, $k=1, \dots, p$.
We now fix an ordered basis
\begin{equation}\label{eq:9}
B= \{X_1(v_1), \dots, X_{r_1}(v_1), w_1, \dots, w_q, v_1, \dots, v_{s_{p_0}}\}
\end{equation}
of
$V$ and let
\begin{align*}
 W &=\operatorname{span}_{\k}\{w_1, \dots, w_q\}, \\
V_0&= \operatorname{span}_{\k}\{v_1, \dots, v_{s_{p_0}}\}.
\end{align*}
We now consider the matrix of a given
$X\in\mathfrak{n}$ with respect to the basis
$B$
$$
\setlength{\unitlength}{5mm}
\begin{picture}(4,4)(1,1)
\linethickness{0.3mm}
\put(0,0){\line(0,1){4}}
\put(8.8,0){\line(0,1){4}}
\put(0,0){\line(1,0){.2}}
\put(0,4){\line(1,0){.2}}
\put(8.8,4){\line(-1,0){.2}}
\put(8.8,0){\line(-1,0){.2}}
\put(0.4,0.25){\small{$A_{3,1}(X)$}}
\put(3.25,0.25){\small{$A_{3,2}(X)$}}
\put(3.25,1.75){\small{$A_{2,2}(X)$}}
\put(3.25,3.25){\small{$A_{1,2}(X)$}}
\put(0.4,1.75){\small{$A_{2,1}(X)$}}
\put(0.4,3.25){\small{$A_{1,1}(X)$}}
\put(6.1,0.25){\small{$A_{3,3}(X)$}}
\put(6.1,1.75){\small{$A_{2,3}(X)$}}
\put(6.1,3.25){\small{$A_{1,3}(X)$}}
\put(-2.5,1.75){\small{$[X]_B=$}}
\scriptsize{
\put(0.4,4){$\overbrace{\rule{30pt}{0pt}}^{r_{1}}$}
\put(3.45,4){$\overbrace{\rule{30pt}{0pt}}^{q}$}
\put(6.1,4){$\overbrace{\rule{30pt}{0pt}}^{s_{p_0}}$}
}
\put(9,3.2){$\left.\rule{0mm}{3.5mm}\right\}r_1$}
\put(9,1.8){$\left.\rule{0mm}{3.5mm}\right\}q$}
\put(9,0.4){$\left.\rule{0mm}{3.5mm}\right\}s_{p_0}$}
\end{picture}
$$

\

\noindent where the row and columns correspond to the decomposition
\[
V=\mathfrak{n}_{1,1}{v_1}\oplus W\oplus V_0.
\]
The following proposition describe the main properties of $[X]_B$.
\begin{proposition}\label{prop.B}
 Let $X\in\mathfrak{n}_{k,j}$ for some $k=1,\dots,p$ and $j=1,\dots,s_k$.
\begin{enumerate}[(1)]
\item If $j=1$ then $\big(A_{1,3}(X)\big)_{h,1}=0$ for all
$h=r_k+1,\dots,r_1$.
In addition,
$\big(A_{1,3}(X)\big)_{h,1}=0$ for all
$h=1,\dots,r_1$ if and only if
$X=0$.

\medskip

\item If $j\ge 2$ then $A_{m,n}(X)=0$ for $m=2,3$, $n=1,2,3$.
On the other hand, the row $A_{1,1}(X)$ $A_{1,2}(X)$ $A_{1,3}(X)$ has the following structure

 \setlength{\unitlength}{7mm}
\begin{picture}(9,9)(0,-1)
\linethickness{0.3mm}
\put(0,0){\line(0,1){7}}
\put(6,0){\line(0,1){7}}
\put(10,0){\line(0,1){7}}
\put(15,0){\line(0,1){7}}
\put(0,0){\line(1,0){.2}}
\put(0,7){\line(1,0){.2}}
\put(15,7){\line(-1,0){.2}}
\put(15,0){\line(-1,0){.2}}
\linethickness{0.1mm}
\put(3,4){\line(0,1){3}}
\multiput(4,3.1)(0,.2){5}{\line(0,1){.1}}
\put(5,2){\line(0,1){1}}
\put(6,1){\line(0,1){1}}
\put(3,4){\line(1,0){1}}
\multiput(4.1,3)(.2,0){5}{\line(1,0){.1}}
\put(5,2){\line(1,0){1}}
\put(6,1){\line(1,0){1}}
\put(7,1){\line(1,0){8}}
\put(12,1){\line(0,1){6}}
\multiput(3.1,6)(.2,0){60}{\line(1,0){.1}}
\multiput(3.1,5)(.2,0){60}{\line(1,0){.1}}
\multiput(4.1,4)(.2,0){55}{\line(1,0){.1}}
\multiput(5.1,3)(.2,0){50}{\line(1,0){.1}}
\multiput(6.1,2)(.2,0){45}{\line(1,0){.1}}
\multiput(4,4.1)(0,.2){15}{\line(0,1){.1}}
\multiput(5,3.1)(0,.2){20}{\line(0,1){.1}}
\scriptsize{
\put(15.3,6.3){$a_{p}$}
\put(15.3,5.3){$\vdots$}
\put(15.3,4.3){$a_{p_0}$}
\put(15.3,3.3){$\vdots$}
\put(15.3,2.3){$a_{k+1}$}
\put(15.3,1.3){$a_{k}$}
\put(15.3,0.3){$r_{1  }-r_{k}$}
%
\put(12.5,7.3){$s_{p_0}-j+1$}
\put(10.6,7.3){$j-1$}
\put(7.5,7.3){$\dim W$}
\put(5.3,7.3){$a_1$}
\put(4.4,7.3){$...$}
\put(4.4,6.3){$...$}
\put(4.2,3.3){$\ddots$}
\put(3.0,7.3){$a_{p_0-k}$}
\put(0.7,7.3){$r_{p_0-k+1}$}
\put(3.4,6.4){$*$}
\put(5.4,6.4){$*$}
\put(8.0,6.4){$*$}
\put(3.4,5.3){$\vdots$}
\put(5.4,5.3){$\vdots$}
\put(8.0,5.3){$\vdots$}
\put(3.4,4.4){$*$}
\put(5.4,4.3){$\vdots$}
\put(8.0,4.3){$\vdots$}
\put(5.4,3.3){$\vdots$}
\put(8.0,3.3){$\vdots$}
\put(5.4,2.4){$*$}
\put(8.0,2.4){$*$}
\put(8.0,1.4){$*$}
\put(13.5,6.4){$*$}
\put(13.5,5.3){$\vdots$}
\put(13.5,4.3){$\vdots$}
\put(13.5,3.3){$\vdots$}
\put(13.5,2.4){$*$}
\put(13.5,1.4){$*$}
\put(2.8,2.4){$0$}
\put(10.8,6.4){$0$}
\put(10.8,5.3){$\vdots$}
\put(10.8,4.3){$\vdots$}
\put(10.8,3.3){$\vdots$}
\put(10.8,2.4){$0$}
\put(10.8,1.4){$0$}
\put(8.0,0.4){$0$}
\put(12.5,0.4){$0$}
\put(0.2,-0.10){$\underbrace{\rule{112pt}{0pt}}_{A_{11}(X)}$}
\put(6.1,-0.10){$\underbrace{\rule{77pt}{0pt}}_{A_{12}(X)}$}
\put(10.1,-0.10){$\underbrace{\rule{97pt}{0pt}}_{A_{13}(X)}$}
}
\end{picture}
\end{enumerate}
where $a_h=r_h-r_{h+1}$,  $h=1,\dots,p-1$ and
$a_p= r_{p}$.
In particular, if  $k\ge p_0$, then $A_{1,1}(X)=0$.
\end{proposition}

\begin{proof}
Part (1) is a consequence of  Theorem \ref{teo:descomposicion}(1) and (4).

If $j\ge 2$, it follows from Theorem \ref{teo:descomposicion}(4) that
$X(v)\in \mathfrak{n}_{k,1}v_1$ for all $v\in V$.
This proves that $\big(A_{1,*}(X)\big)_{h,*}=0$ for all $h\ge k$.
It follows from Theorem \ref{teo:descomposicion}(3) that
$\big(A_{1,3}(X)\big)_{*,h}=0$ for all $h\le j-1$.

Finally, let us prove that $A_{1,1}(X)$ has the staircase-shape stated above.
If $i=1,\dots,r_1$, 
then $i^\text{th}$ element of $B$ is $X_i(v_1)$.
If additionally $i\le r_h$, for some $h=1, \dots, p$, then
$X_i\in\n_{h,1}$ and since
$X\in\n_{k,j}$ we obtain
$[X,X_i]\in\n_{k+h}$.
Thus
\begin{align*}
 X X_i(v_1)&=X_i X(v_1) + [X,X_i](v_1) \\
	  &=[X,X_i](v_1) \in \n_{k+h}v_1 \qquad\text{(since $j\ge2$)}.
\end{align*}
This completes the proof.
\end{proof}

\medskip

\noindent
\textbf{Question.} Since $\mathfrak{n}\subset\gl(V)$ consists of nilpotent endomorphisms,
it would be very interesting to obtain a basis $B$ such that $[X]_B$
is upper triangular for all $X\in\n$,
in addition to the properties stated in Proposition \ref{prop.B} (or similar ones).
This would transform Proposition \ref{prop.B} into a detailed version of Lie's Theorem
that takes into account a given filtration of the Lie algebra $\n$.
As stated, Proposition \ref{prop.B} is enough to obtain the lower bounds that we are looking for.

\medskip

\begin{theorem}\label{lemma:maximo2}
Let
$\mathfrak{n}$ be a Lie subalgebra of nilpotent operators of
$\mathfrak{gl}(V)$ and let
$\mathfrak{n}_p \subset \dots \subset \mathfrak{n}_1=\mathfrak{n}$ be a filtration of
$\mathfrak{n}$ such that $\mathfrak{n}_{p_0}$ is contained in the center of $\n$.
 Then there exists integers $a_k\ge0$, $k=0,\dots,p$, with $a_0,a_p\ge1$, such that:
\begin{enumerate}[(1)]
\item $\displaystyle \dim \mathfrak{n}_k \leq \sum_{i=0}^{p_0-k} a_i \left(a_{k + i} + \dots + a_p\right)$ for
    $k= 1, \dots, p_0$.
    \medskip
\item $\dim \mathfrak{n}_k \leq  a_0 \left(a_{k } + \dots + a_p\right)$ for
    $k= p_0, \dots, p$.
    \bigskip
\item $\dim V= a_0 + a_1 + \dots + a_p$.
\end{enumerate}
\end{theorem}

\begin{proof}
Let $B$ be the basis of $V$ as in \eqref{eq:9} and let $T:\n\to\gl(V)\oplus V$ be defined by
\[
 T(X)=\begin{cases}
       X(v_1)\in V, & \text{ if $X\in\n_{1,1}$;} \\[2mm]
       X\in\gl(V), & \text{ if $X\in\n_{1,j}$, $j\ge2$.}
      \end{cases}
\]
It follows from  Theorem \ref{teo:descomposicion} that $T$ is injective.
We apply Proposition \ref{prop.B} to obtain a bound for $\dim T(\n_k)$.
On the one hand, we know from Theorem \ref{teo:descomposicion}(2) that
\[
 \dim T(\n_{k,1})=r_k.
\]
On the other hand, from Proposition \ref{prop.B}(2), when $j\ge 2$,
we know the shape of the matrices
$[T(\n_{k,j})]_B$.
Taking into account that the first column of $A_{1,3}(T(\n_{k,j}))$ is zero if
$j\ge 2$,
we obtain
\[
\dim T\Big( \bigoplus_{j=2}^{s_k} n_{k,j}\Big) \le
\underbrace{(\dim W\!+\!\s_{p_0})r_{k}\!-\!r_k}_{
{\begin{matrix} \\[-6mm]
\scriptscriptstyle\text{size of $A_{1,2}$ and $A_{1,3}$}\\[-1mm]
\scriptscriptstyle\text{except the $1^{\text{st}}$ column of $A_{1,3}$}
\end{matrix}
}}
+
\underbrace{a_1r_{k+1}+a_2r_{k+2}+\dots+a_{p_0-k}r_{p_0}}_{
{\begin{matrix} \\[-6mm]
\scriptscriptstyle\text{size of the staircase in $A_{1,1}$,}\\[-1mm]
\scriptscriptstyle\text{it appears only if $k< p_0$}
\end{matrix}
}}
\]
where
$a_h=r_h-r_{h+1}\ge 0$,
$h=1,\dots,p-1$.
Therefore
\[
\dim T(n_{k}) \le
\begin{cases}
\displaystyle(\dim W+\s_{p_0})r_{k}+\sum_{i=1}^{p_0-k} a_ir_{k+i},  & \text{ if $k<p_0$;} \\[2mm]
(\dim W+\s_{p_0})r_{k}, & \text{ if $k\ge p_0$.}
\end{cases}
\]
If
$a_p=r_p\ge 1$ and
$a_0=\dim W+\s_{p_0}\ge 1$, then
rewriting the above inequality in terms of
${a_k}'s$, we obtain
\[
\dim T(n_{k}) \le
\begin{cases}
\displaystyle a_0(a_{k}+\dots+a_p)+\sum_{i=1}^{p_0-k} a_i(a_{k+i}+\dots+a_p),  & \text{ if $k<p_0$;} \\[2mm]
 a_0(a_{k}+\dots+a_p), & \text{ if $k\ge p_0$.}
\end{cases}
\]
This shows (1) and (2).

Finally $a_{0}+a_1+\dots+a_p=\dim W+\s_{p_0}+r_1=\dim V$.
\end{proof}

Theorem \ref{lemma:maximo2} leads us to consider the following optimization problem.

\begin{problem}\label{prob.opt}
Given integer numbers $p\ge p_0 \ge 1$ and
$n_1,\dots, n_p$, let
\[
 r_k = a_k+a_{k+1}+ \dots +a_p,\quad a_0,a_1,\dots ,a_p\in \Z,
\]
for $k= 0, \dots, p$.
Find the minimum value $r_0^{\text{min}}$ of
\[
 r_0 = a_0+a_1+ \dots +a_p,
\]
subject to the following restrictions

\medskip
\noindent
\begin{tabular}{lll}

(a) & $a_0,a_p\ge 1$ and $a_k\ge 0$, & for $k=1,\dots,p-1$;  \\[2mm]

(b) & $\displaystyle\sum_{i=0}^{p_0-k} a_i r_{k + i}\ge {n}_k $, & for  $k= 1, \dots, p_0$; \\[6mm]

(c) & $a_0 r_{k }  \geq  {n}_k$, & for    $k= p_0, \dots, p$.
\end{tabular}
\end{problem}

\medskip

The solution to this problem gives us a lower bound for $\mu_{nil}$.

\begin{corollary}\label{cor.main}
Let
$\mathfrak{n}$ be a Lie algebra and let
$\mathfrak{n}_p \subset \dots \subset \mathfrak{n}_1=\mathfrak{n}$ be a filtration of
$\mathfrak{n}$ such that $\mathfrak{n}_{p_0}$ is contained in the center of $\n$.
Then
\[
 \mu_{nil}(\mathfrak{n})\ge r_0^{\text{min}}
\]
where 
$r_0^{\text{min}}$ is the minimum value of Problem \ref{prob.opt} associated to
$p\ge p_0 \ge 1$ and 
$n_k=\dim \mathfrak{n}_k$ with
$k= 1, \dots, p$.
\end{corollary}

The optimization problem above seems to be difficult and, in this paper,
we will just give a pair of quick, but not trivial, estimates of its solution.
In the last section we discuss two simplifications of Problem \ref{prob.opt}
that lead respectively to Theorem \ref{cota1} and Theorem \ref{proplower} below.
We think that it is worth studying  Problem \ref{prob.opt} in more detail in the future
to obtain more accurate results than the following two theorems.

\begin{theorem}[First simplification]\label{cota1}
Let
$\mathfrak{n}$ be a nilpotent Lie algebra and let
$\mathfrak{n}_p \subset \dots \subset \mathfrak{n}_1=\mathfrak{n}$ be a filtration of
$\mathfrak{n}$ such that $\mathfrak{n}_{p_0}$ is contained in the center of 
$\n$ for some $p_0=1,\dots, p$.
Then
\[
 \mu_{nil}(\mathfrak{n})\ge \sqrt{\frac{2(p_0+1)}{p_0} \dim \mathfrak{n}}.
\]
In particular, if 
$\mathfrak{n}$ is 
$p$-step nilpotent Lie algebra then
$
 \mu_{nil}(\mathfrak{n})\ge \sqrt{\frac{2(p+1)}{p} \dim \mathfrak{n}}.
$
\end{theorem}

\begin{theorem}[Second simplification]\label{proplower}
Let
$\mathfrak{n}$ be a nilpotent Lie algebra and let
$\mathfrak{n}_p \subset \dots \subset \mathfrak{n}_1=\mathfrak{n}$ ($p>1$) be a filtration of
$\mathfrak{n}$ such that $\mathfrak{n}_{p_0}$ is contained in the center of 
$\n$ for some $p_0=2,\dots, p$ and let 
$n_i= \dim \n_i$,  
$i= 1, \dots, p$.
\medskip
\begin{enumerate}[(1)]
 \item If ${n_1}\ge\big((p_0-1)^2+p_0^2 \big){n_{p_0}}$ then
 \[
   \mu_{nil}(\mathfrak{n})\ge
 \sqrt{\frac{2p_0}{p_0-1}\big(n_1-n_{p_0})}.
 \]
 \item If ${n_1}\le\big( (p_0-1)^2+p_0^2 \big){n_{p_0}}$ then
\[
  \mu_{nil}(\mathfrak{n})\ge
\sqrt{\frac{2(p_0\!-\!1)}{p_0\!-\!2}n_1+\frac{2p_0(p_0\!-\!1)}{(p_0\!-\!2)^2}n_{p_0}}
-\frac{2}{p_0\!-\!2}\sqrt{n_{p_0}},
\]
if $p_0\ne 2$, and   $ \mu_{nil}(\mathfrak{n})\ge \frac{n_1+3n_2}{2\sqrt{n_2}}$, if $p_0=2$.
\end{enumerate}
In both cases, the given bound is bigger than 
$\sqrt{\frac{2(p_0+1)}{p_0}n_1}$.
\end{theorem}
Both results are proved in \S\ref{sec.estimates}. 
Although Theorem \ref{Thm.mainbound} is an immediate corollary of Theorem \ref{proplower} they will be treated
separately since it is much easier to obtain directly Theorem \ref{Thm.mainbound} from Corollary \ref{cor.main}.
This will also show some of the difficulties involved in Problem \ref{prob.opt}.


\section{Some applications}\label{sec.aplications}


\begin{enumerate}[(1)]
\item Given
$p, a \in \mathbb{N}$, let
\begin{equation*}\label{eq:2}
{\mathfrak{n}}_{a,p} = \left\{
          \left(
          \begin{smallmatrix}
          0 & A_{12} & A_{13} & \dots & A_{1p+1} \\
            & 0 & A_{23} & \dots & A_{2p+1} \\
            &   &   \ddots     &       & \vdots\\
            &  0  &        &       & A_{pp+1} \\
            &    &        &       & 0
          \end{smallmatrix}
          \right) : A_{ij} \in M_a(\k) \text{ para } 1 \leq i < j \leq p+1
          \right\}.
\end{equation*}
It is clear that
$\mathfrak{n}_{a,p}$ is a $p$-step nilpotent Lie subalgebra of
$\mathfrak{sl}((p+1)a,\k)$ and $\dim\mathfrak{n}_{a,p}=\frac{(p + 1)p}{2}a^2$.
Its defining representation has dimension  $(p + 1)a$.
Since Theorem \ref{cota1} states that
\[
 \mu(\mathfrak{n}_{a,p}) \geq \sqrt{\frac{2(p+1)}{p} \dim\mathfrak{n}_{a,p}}
           = (p + 1)a,
\]
we obtain  $\mu(\mathfrak{n}_{a,p}) = (p + 1)a$.
\item Given
$a, b, c \in \mathbb{N}$  let
$$
\mathfrak{n}_{a,b,c}= \left\{
            \left(\begin{smallmatrix}
          0 & A_{ab} & A_{ac}\\
            & 0 & A_{bc}\\
          &   & 0
          \end{smallmatrix}
          \right) : A_{ab} \in M_{a,b}(\k),A_{ac} \in M_{a,c}(\k), A_{bc}(\k) \in M_{b,c}\right\}.
$$
Now
$\mathfrak{n}_{a,b,c}$ is a 
$2$-step nilpotent Lie subalgebra of
$\mathfrak{sl}(a + b + c,\k)$ and $\dim\mathfrak{n}_{a,p}=ab+bc+ac$.
The center of 
$\mathfrak{n}_{a,b,c}$  is the dimension 
$ac$.

If $b=a+c$ then we are under the conditions stated in part (1),  Theorem \ref{Thm.mainbound}
and the given lower bound for $\mu(\mathfrak{n}_{a,b,c})$ coincides with the dimension of the defining
representation of  $\mathfrak{n}_{a,b,c}$.

If  $a=c$ and $b\le 2a$ then we are under the conditions stated in (2), Theorem \ref{Thm.mainbound}
and the given lower bound for $\mu(\mathfrak{n}_{a,b,c})$ coincides with the dimension of the defining
representation of  $\mathfrak{n}_{a,b,c}$.

Thus, if either 
$b=a+c$, or 
$a=c$ and 
$b\le 2a$, we have  
$$
\mu(\mathfrak{n}_{a,b,c}) = a+b+c.
$$

We point out that in some cases $\mu(\mathfrak{n}_{a,b,c}) < a+b+c$.
For instance, it is shown in \cite{AR} that
$$
\mu(\mathfrak{n}_{1,1,c})= \left \lceil 2\sqrt{2c} \right\rceil < 1 + 1 + c.
$$
for all $c\in\N$.
\end{enumerate}


\section{Estimates for the solution of Problem \ref{prob.opt}}\label{sec.estimates}


In this section we will show some bounds for $r_0^{\text{min}}$ resulting from
considering Problem \ref{prob.opt} with real (instead of integer) variables.
Since $r_0$ is linear, it is clear that, in this case, $r_0^{\text{min}}$ will
be reached in a boundary point of the restriction set.

\subsection{A first simplification}
It is clar that $r_0^{\text{min}}$ is greater than or equal to the minimum of
\begin{equation*}
 r_0 = a_0+\dots+a_{p}
\end{equation*}
subject to

\medskip
\noindent
\begin{tabular}{lll}
(a') &  $a_0,a_{p}>0$ and $a_k\in\mathbb{R}_{\ge 0}$, & for $k=0,\dots,p$; \\[2mm]
(b') & $\displaystyle\sum_{i=0}^{p_0-1} a_i r_{1 + i}\ge {n}_1 $; \\[5mm]
\end{tabular}

\noindent
We will think $r_{p_0}$ as an independent variable in this problem and thus we can reformulate it looking for a minimum of 
\begin{equation*}
 r_0 = a_0+\dots+a_{p_0-1}+r_{p_0}
\end{equation*}
subject to

\medskip
\noindent
\begin{tabular}{lll}
(a') &  $a_0,r_{p_0}>0$ and $a_k\in\mathbb{R}_{\ge 0}$, & for $k=0,\dots,p_0-1$; \\[2mm]
(b') & $\displaystyle\sum_{i=0}^{p_0-1} a_i r_{1 + i}\ge {n}_1 $; & (here $r_i=a_i+\dots + a_{p_0-1}+r_{p_0}$)\\[5mm]
\end{tabular}

\noindent
Let us call this problem as Problem (a'b') for $p_0$.
We notice that, if we consider Problem (a'b') for $p_0$ with the additional restriction $a_k=0$ for some $k=1,\dots,p_0-1$,
then the problem becomes Problem (a'b') for $p_0-1$.
Therefore, in order to find the minimum value of $r_0$ we may consider  $a_k> 0$ for all $k=0,\dots,p_0-1$.
Moreover, since the minimum will be reached at the boundary,  
we can reformulate Problem (a'b') for $p_0$ as: find the minimum of $r_0$
subject to

\medskip
\noindent
\begin{tabular}{lll}
(a') &  $a_k,r_{p_0}>0$,  & for $k=0,\dots,p_0-1$; \\[2mm]
(b') & $\displaystyle\sum_{i=0}^{p_0-1} a_i r_{1 + i}= {n}_1 $.\\[5mm]
\end{tabular}

\medskip
We now will find the minimum value of $r_0$ in this problem. 
We use (b') in order to eliminate the variable $a_{0}$.
Thus we will think $r_0$ as a function of $a_1,\dots,a_{p_0-1},r_{p_0}$,
and we will find the critical values of $r_0$, and next its minimum.

It is not difficult to see that the only critical value of $r_0$ is
\[
(a_1,\dots,a_{p_0-1},r_{p_0}) = (a_0,a_0,\dots,a_0,a_0)
\]
with
$n_1 = \frac{p_0(p_0+1)}2\;a_0^2$.
Also, it is not difficult to see that this is a local minimum and it yields
\begin{equation*}
 r_0=\sqrt{\frac{2(p_0+1)}{p_0}n_1}.
\end{equation*}
This is in fact a global minimum.
Indeed, since $\sqrt{\frac{2(p_0+1)}{p_0}n_1}$ is decreasing as a function of $p_0$, 
taking into account the remark explained above, 
we can not obtain smaller values of $r_0$ by allowing $a_k=0$ for some $k$.

\subsection{A second simplification} 
In this case, we can do an analysis similar to what we did in the first simplification to conclude that 
 $r_0^{\text{min}}$ is  greater than or equal to the minimum of
\begin{equation}\label{eq.a_{0}}
 r_0 = a_0+\dots+a_{p_0-1}+r_{p_0}
\end{equation}
subject to

\medskip
\noindent
\begin{tabular}{lll}
(a') &  $a_k,r_{p_0}>0$,  & for $k=0,\dots,p_0-1$; \\[2mm]

(b') & $\displaystyle\sum_{i=0}^{p_0-1} a_i r_{1 + i}= {n}_1 $; \\[5mm]

(c') & $a_0 r_{p_0}   =  {n}_{p_0}$, \\[3mm]
\end{tabular}

\noindent
In this case we will use (b') and  (c') in order to eliminate the variables $r_{p_0}$ and $a_{p_0-1}$.
Thus we will think $r_0$ as a function of $a_0,\dots,a_{p_0-2}$, we will find its critical values,
and its minimum.

It follows from (c') that
\begin{equation}\label{eq.partial_r_{p_0}}
 \frac{ \partial r_{p_0}}{\partial a_j}   =
 \begin{cases}
- \displaystyle\frac{r_{p_0}}{a_{0}}, & \text{$j=0$;} \\[3mm]
 0, & \text{$1\le j\le p_0\!-\!2$.}
 \end{cases},
 \quad
 \frac{ \partial^2 r_{p_0}}{\partial a_i a_j}   =
 \begin{cases}
 \displaystyle\frac{2r_{p_0}}{a_{0}^2}, & \text{$i=j=0$;} \\[3mm]
 0, & \text{$1\le i,j\le p_0\!-\!2$.}
 \end{cases}
\end{equation}
It follows from (b') that
\begin{equation}\label{eq.a_{p_0-1}}
  n_{1} = (r_0-r_{p_0-1})(a_{p_0-1}+r_{p_0}) + a_{p_0-1} r_{p_0} + \sum_{i=0}^{p_0-3} a_i(r_{i+1}-r_{p_0-1}),
\end{equation}
and we obtain from  \eqref{eq.a_{p_0-1}}
\begin{equation}\label{eq.partial_a_{p_0-1}}
 \frac{ \partial a_{p_0-1}}{\partial a_j}   =
  \begin{cases}
  -  \displaystyle\frac{(r_{0}-a_0-r_{p_0})(a_{0}-r_{p_0})}{(r_0-a_{p_0-1})a_{0}}  , & \text{$j=0$;}     \\[5mm]
  - \displaystyle\frac{r_{0}-a_j}{r_0-a_{p_0-1}}, &  \text{$1\le j\le p_0-2$.}
  \end{cases}
\end{equation}
and
\begin{equation}\label{eq.partial2_a_{p_0-1}}
 \frac{ \partial^2 a_{p_0-1}}{\partial a_i a_j}   =
  \begin{cases}
    \displaystyle\frac{2(r_{0}\!-\!a_0\!-\!r_{p_0})(a_0^2\!-\!r_{p_0}(2a_{0}\!+\!r_0\!-\!r_{p_0\!-\!1}))}{(r_0\!-\!a_{p_0\!-\!1})^2 a_{0}^2},
                                                              & \text{$i=j=0$;}     \\[5mm]
    \displaystyle\frac{(r_{0}\!+\!a_{p_0\!-\!1}\!-\!a_0\!-\!a_i\!-\!r_{p_0})(a_{0}\!-\!r_{p_0})}{(r_0\!-\!a_{p_0\!-\!1})^2a_{0}}  ,
                                                              & \text{$0=j<i\le p_0\!-\!2$;}     \\[5mm]
    \displaystyle\frac{(r_{0}\!+\!a_{p_0\!-\!1}\!-\!a_j\!-\!a_i)}{(r_0\!-\!a_{p_0\!-\!1})^2}  ,
                                                              & \text{$0<j<i\le p_0\!-\!2$;}     \\[5mm]
  \displaystyle\frac{2(r_{0}\!-\!a_j)}{(r_0\!-\!a_{p_0\!-\!1})^2}  ,  &  \text{$0<j=i\le p_0\!-\!2$;}
  \end{cases}
\end{equation}
Therefore, it follows from \eqref{eq.a_{0}}, \eqref{eq.partial_r_{p_0}} and \eqref{eq.partial_a_{p_0-1}} that
 \[
  \frac{ \partial r_{0}}{\partial a_j}   =
  \begin{cases}
    \displaystyle\frac{(r_{p_0} +a_0-a_{p_0-1})(a_{0}-r_{p_0})}{(r_0-a_{p_0-1})a_{0}}  , & \text{$j=0$;}     \\[5mm]
    \displaystyle\frac{a_j-a_{p_0-1}}{r_0-a_{p_0-1}}, &  \text{$1\le j\le p_0-2$.}
  \end{cases}
 \]

 The (possible) critical values of $r_0$ are two.
 First
\[
(a_0,a_1,a_2,\dots,a_{p_0-1},r_{p_0}) =
(a_0,a_1,a_1,\dots,a_1,a_0)
\]
with
\begin{align*}
n_1 &= \frac{(p_0-1)(p_0-2)}2\;a_1^2 + 2(p_0-1)\;a_0a_1 +a_0^2, \\
n_{p_0}&= a_0^2;
\end{align*}
whose positive solutions are $a_0 =\sqrt{n_{p_0}}$ and
\begin{align*}
a_1 &= \frac{\sqrt{2(p_0-1)\big((p_0-2)n_1+p_0n_{p_0}\big)}-2(p_0-1)\sqrt{n_{p_0}}}{(p_0-1)(p_0-2)}.
\end{align*}
This yields
\begin{equation}\label{eq.r_0_maxmin}
 r_0=\sqrt{\frac{2(p_0-1)}{p_0-2}n_1+\frac{2p_0(p_0-1)}{(p_0-2)^2}n_{p_0}}-\frac{2}{p_0-2}\sqrt{n_{p_0}}.
\end{equation}
This critical value always exists.
The second case is
\[
(a_0,a_1,a_2,\dots,a_{p_0-1},r_{p_0})  =
(a_0,a_1,a_1,\dots,a_1,a_1-a_0)
\]
with
\begin{align*}
n_1 &= \frac{p_0(p_0-1)}2\;a_1^2 + a_0a_1 -a_0^2, \\
n_{p_0}&=n_1 - \frac{p_0(p_0-1)}2\; a_1^2;
\end{align*}
whose positive solutions are $a_1 = \sqrt{\frac{2(n_1-n_{p_0})}{p_0(p_0-1)}}$ and
\begin{align*}
a_0 &=\sqrt{\frac{n_1-n_{p_0}}{2p_0(p_0-1)}}\pm
\sqrt{\frac{n_1-n_{p_0}}{2p_0(p_0-1)}-n_{p_0}}.
\end{align*}
These $\pm$ critical values exist if and only if
\begin{equation}\label{eq.condition}
 n_1\ge \big((p_0-1)^2+p_0^2\big)n_{p_0}
\end{equation}
and either of them yields
\begin{equation}\label{eq.r_0_min}
 r_0=\sqrt{\frac{2p_0}{p_0-1}(n_1-n_{p_0})}.
\end{equation}

If condition  \eqref{eq.condition} holds, then
the value of \eqref{eq.r_0_min} is a local minimum
(and the value of \eqref{eq.r_0_maxmin} is a local maximum).
If condition  \eqref{eq.condition} does not hold, then
the value of \eqref{eq.r_0_maxmin} is a local minimum.

Arguing as we did with in first simplification we conclude that  
the value of \eqref{eq.r_0_min}, if \eqref{eq.condition} holds,
and the value of \eqref{eq.r_0_maxmin}, 
if \eqref{eq.condition} does not hold,
is a global minimum.


\end{document}